\documentclass[a4paper,10pt]{article}
\usepackage{amssymb}
\usepackage{amsthm}
\usepackage{amsmath}
\usepackage{url}

\usepackage{longtable}

\newtheorem{theorem}{Theorem}[section]
\newtheorem{lemma}[theorem]{Lemma}

\newtheorem{definition}[theorem]{Definition}
\newtheorem{remark}[theorem]{Remark}

\date{}
\author{Philipp Grohs\footnote{TU Wien, Institute of Discrete Mathematics and Geometry, Wiedner Hauptstrasse 8-10, 1040 Wien, Austria. Email: philippgrohs@gmail.com, web: http://www.dmg.tuwien.ac.at/grohs, phone: +43 58801 11318.}}
\title{Continuous Shearlet Tight Frames}
\begin{document}
\maketitle
\begin{abstract}
Based on the shearlet transform we present a general construction of continuous tight frames for $L^2(\mathbb{R}^2)$ from any sufficiently smooth 
function with anisotropic moments.
This includes for example compactly supported systems, piecewise polynomial systems, or both. From our earlier results in \cite{Grohs2010a} it follows 
that these systems enjoy the same desirable approximation properties for directional data as the previous bandlimited and very specific
constructions due to Kutyniok and Labate; \cite{Labate2009}. 
We also show that the representation formulas
we derive are in a sense optimal for the shearlet transform.
\end{abstract}
\tableofcontents
\newcommand{\Z}{{\mathbb Z}}
\newcommand{\N}{{\mathbb N}}
\newcommand{\R}{{\mathbb R}}
\newcommand{\schoen}{\mathcal}
\newcommand{\V}{{\schoen V}}
\newcommand{\F}{{\schoen F}}
\newcommand{\T}{{\schoen T}}
\newcommand{\A}{{\schoen A}}
\newcommand{\p}{{p}}
\newcommand{\q}{{\bf q}}
\newcommand{\h}{{\bf h}}
\newcommand{\e}{{\bf e}}
\newcommand{\Pp}{{\schoen P}}
\newcommand{\Ss}{{S}}
\newcommand{\Uu}{{\schoen U}}
\newcommand{\ii}{{\bf i}}
\renewcommand{\j}{{\bf j}}
\renewcommand{\k}{{\bf k}}
\renewcommand{\d}{{\bf d}}
\renewcommand{\l}{{\bf l}}
\newcommand{\C}{\mathcal{C}}
\newcommand{\D}{\mathcal{D}}
\section{Introduction}
The purpose of this short paper is to give a construction of a system of bivariate functions which has the following
desirable properties:
\begin{enumerate}
	\item[]{\bf directionality.} The geometry of the set of singularities of a tempered distribution $f$ can be accurately described
	in terms of the interaction between $f$ and the elements of the system. 
	\item[]{\bf tightness.} The system forms a tight frame of $L^2(\R^2)$.
	\item[]{\bf locality.} The representation is \emph{local}. By this we mean that the representation can also be interpreted
		as a representation with respect to a non tight frame and its dual frame such that both of these frames only consist of compactly
		supported functions.
\end{enumerate}
The construction is based on the \emph{shearlet transform} which has been introduced in \cite{Labate2005} and has become popular
in Computational Harmonic Analysis for its ability to sparsely represent bivariate functions. A related important result is that the coefficients
of a general tempered distribution with respect to this transform exactly characterize the \emph{Wavefront Set} (see e.g.  
\cite{Hormander1983} for the definition) of this distribution \cite{Labate2009}. \\
{\bf Notation.} We shall use the symbol $|\cdot |$ indiscriminately for the absolute value on $\R, \R^2, \mathbb{C}$ and $\mathbb{C}^2$. We usually denote 
vectors in $\R^2$ by $x,t,\xi,\omega$ and their elements by $x_1,x_2,t_1,t_2,\dots$. In general it should always be clear to which space a variable belongs. 
The symbol $\|\cdot\|$ is reserved for various function space and operator norms. For two vectors $s,t\in \R^2$ we denote by $st$ their Euclidean inner product. 
We use the symbol $f\lesssim g$ for two functions $f,g$ if there exists a constant $C$ such that $f(x)\le Cg(x)$ for large
values of $x$. We will often speak of \emph{frames}. By this we mean continuous frames as defined in \cite{Ali1993}. We define $\hat f(\omega)=\int f(x)\exp(2\pi i \omega x) dx$
to be the Fourier transform for a function $f\in L^1\cap L^2$ and continuously extend this notion to tempered distributions. 
\subsection{Shearlets}  

We start by defining what a shearlet is and what the shearlet transform is:
\begin{definition}
  	A function $\psi\in L^2(\R^2)$ is called \emph{shearlet} if it possesses $M\geq 1$ vanishing moments
	in $x_1$-direction, meaning that 
	$$
		\int_{\R^2}\frac{|\hat \psi (\omega)|^2}{|\omega_1|^{2M}}d\omega <\infty.
	$$
Let $f\in L^2(\R^2)$. The \emph{shearlet transform} of $f$ with respect to a shearlet $\psi$ maps $f$ to
$$
	\mathcal{SH}_\psi f(a,s,t):=\langle f ,  \psi_{ast}\rangle,
$$
where
$$
	\psi_{ast}(x):=\psi\big(\frac{x_1 - t_1 - s(x_2 - t_2)}{a},\frac{x_2 - t_2}{a^{1/2}}\big).
$$ 
\end{definition}
The following reproduction formula holds \cite{Dahlke2008}:
\begin{theorem}
	For all $f\in L^2(\R^2)$ and $\psi$ a shearlet
	\begin{equation}\label{eq:repr1}
		f(x) = \int_{\R^2}\int_{\R}\int_{\R_+}\mathcal{SH}_\psi(a,s,t) \psi_{ast}(x) a^{-3}dadsdt,
	\end{equation}
	where equality is understood in a weak sense.
\end{theorem}

The shearlet transform captures local, scale- and directional information via the parameters $t,a,s$ respectively. A significant 
drawback of this representation is the fact that the directional parameter runs over the non compact set $\R$. Also it is easy
to see that the distribution of directions becomes infinitely dense as $s$ grows. These problems led to the construction
of shearlets on the cone \cite{Labate2009}. The idea is to restrict the shear parameter $s$ to a compact interval. Since this only allows
to caption a certain subset of all possible directions, the function $f$ is split into $f=Pf + P^\nu f$, where $P$ is the frequency
projection onto the cone with slope $\le 1$ and only $Pf$ is analyzed with the shearlet $\psi$ while $P^\nu f $ is analyzed
with $\psi^\nu(x_1,x_2):= \psi(x_2,x_1)$. For very specific choices of $\psi$ Labate and Kutyniok proved a representation formula

\begin{eqnarray}\nonumber
	\|f\|_2^2 &=& \int_{\R^2}|\langle f, T_t W\rangle |^2 dt + \int_{\R^2}\int_{-2}^2\int_0^1 |\mathcal{SH}_\psi Pf(a,s,t)|^2 a^{-3}da ds dt \\
			&& +  \int_{\R^2}\int_{-2}^2\int_0^1 |\mathcal{SH}_{\psi^\nu} P^\nu f(a,s,t)|^2 a^{-3}da ds dt,\label{eq:tforig}
\end{eqnarray}
$T_t$ denoting the translation operator by $t$. The main weakness of this decomposition is its lack of locality: indeed, first of all the 
need to perform the frequency projection $P$ to $f$ destroys any locality. But also the functions $\psi$ which have been considered
in $\cite{Labate2009}$ are very specific bandlimited functions which do not have compact support. As a matter of fact no useful local representation
via compactly supported functions which is able to capture directional smoothness properties has been found to date. The present
paper aims at providing a step towards finding such a representation using two crucial observations:
First, in \cite{Grohs2010a} we were able to show that the description of directional smoothness via the decay rate of the shearlet coefficients for $a\to 0$
essentially works for any function which is sufficiently smooth and has sufficiently many vanishing moments in the first direction, hence also for 
compactly supported functions. The second observation is that actually a full frequency projection $P$ is not necessary to arrive at a useful 
representation similar to (\ref{eq:tforig}).   Instead of the operator $P$ we will use a 'localized frequency projection' which is given by Fourier multiplication
with a function $\hat p_0$ to be defined later.

Our main result Theorem \ref{thm:main} will prove a representation formula similar to (\ref{eq:tforig}) where $\psi$ is allowed
to be compactly supported
and the frequency projection is replaced with a local variant. We also show that this local variant of the frequency projection is
the best we can do -- without it, no useful continuous tight frames (or continuous  frames with a structured dual) can be constructed within the scope of the shearlet
transform. This is shown in 
Theorem \ref{thm:nonex}.
\section{The construction}
The goal of this section is to derive a representation formula
\begin{eqnarray}\nonumber
		\|f\|_2^2 &=& \frac{1}{C_\psi}\big(\int_{\R^2}\int_{-2}^2\int_0^1 |\langle f , q_0\ast \psi_{ast}\rangle|^2 a^{-3}da ds dt \\\nonumber
		          &+& \int_{\R^2}\int_{-2}^2\int_0^1 |\langle f , q_1\ast \psi_{ast}^\nu\rangle|^2 a^{-3}da ds dt \\\label{eq:repr1}
				  &+& \int_{\R^2}|\langle f , T_t \varphi\rangle|^2dt \big)
\end{eqnarray}
for $L^2(\R^2)$-functions $f$ and with some localized frequency projections $q_0,q_1$ and a window function $\varphi$ to be defined 
later. In order to guarantee that the shearlet-part contains the high-frequency part of $f$ and the 
rest contains low frequencies, it is necessary to ensure that the window function $\varphi$ in this formula is sufficiently smooth. 
\subsection{Ingredients}
Here we first state the definitions and assumptions that we use in the construction. Then we collect some
auxiliary results which we later combine to prove our main results Theorems \ref{thm:main} and \ref{thm:nonex}.
The large part of the results will concern the construction of a useful window function and to ensure its smoothness.
We say that a bivariate function $f$ has Fourier decay of order $L_i$ in the $i$-th variable ($i\in \{1,2\}$) if
$$
	\hat f(\xi) \lesssim |\xi_i|^{-L_i}.  
$$

We start with a shearlet $\psi$ which has $M$ vanishing moments in $x_1$-direction and
Fourier decay of order $L_1$ in the first variable. It is clear from the definition of vanishing
moments that $\psi=\big(\frac{\partial}{\partial x_1}\big)^M \theta$ with $\theta \in L^2(\R^2)$.
We assume that $\theta$ has Fourier decay of order $L_2$ in the second variable so
that the following relation holds:

\begin{equation}\label{eq:momsm}
	2M-1/2 > L_2 > M \geq 1.
\end{equation}
We also set 

\begin{equation}
	N:=2\min(L_2 - M,L_1),
\end{equation}

\begin{equation}
	C_\psi := \int_{\R^2}\frac{|\hat\psi (\omega)|^2}{|\omega_1|^2} d\omega,
\end{equation}
and
\begin{equation}
	\Delta_\psi(\xi):=\int_{-2}^{2}\int_0^1 |\hat \psi(a\xi_1,a^{1/2}(\xi_2-s\xi_1))|^2a^{-3/2}da ds.
\end{equation}
We define functions $\varphi_0,\ \varphi_1$ via
\begin{equation}
	|\hat \varphi_0(\xi)|^2 = C_\psi - \Delta_\psi(\xi)\quad \mbox{  and  }\quad |\hat \varphi_1(\xi)|^2 = C_\psi - \Delta_{\psi^\nu}(\xi),
\end{equation}
where
\begin{equation*}
	\psi^\nu(x_1,x_2):= \psi(x_2,x_1).
\end{equation*}
We write $\chi_\C$ for the indicator function of the cone $\C=\{\xi =(\xi_1,\xi_2)\in \R^2 : |\xi_2|\le |\xi_1|\}$. 
Finally, we pick a smooth and compactly supported bump function $\Phi$ with  $\Phi(0)=1$ and define functions
$p_0,\ p_1$ via
\begin{equation}
	\hat p_0 = \hat\Phi \ast \chi_\C\quad \mbox{  and  }\hat p_1 = 1 - \hat p_0.
\end{equation} 
Clearly, $p_0$ and $p_1$ are both compactly supported tempered distributions.

\begin{lemma}\label{lem:dec1}
	We have 
	\begin{eqnarray}\label{eq:pdecay}\nonumber
		|\hat p_0(\xi)|& \lesssim & |\xi |^{-N}\quad \mbox{for }|\xi_1|/|\xi_2| \geq 3/2 \\  
		|\hat p_1(\xi)|& \lesssim & |\xi |^{-N}\quad \mbox{for }|\xi_2|/|\xi_1| \geq 3/2 
	\end{eqnarray}
\end{lemma}

\begin{proof}
	Assume that $\xi=te$ with $e$ a unit vector with $|e_1|/|e_2|\geq 3/2$ and $t>0$. There exists a uniform $\delta>0$ such that
	for all $\eta$ with $|\eta|<\delta t$ we have $\xi - \eta \in \C^c$, and hence $\chi_\C(\xi  - \eta)=0$. 
	It follows that we can write
	\begin{eqnarray*}
		|\hat p_0(\xi)| & = & |\int_{\R^2}\chi_\C(\xi - \eta)\hat \Phi(\eta) d\eta| \le  \int_{|\eta|>\delta t}|\hat \Phi(\eta)| d\eta \\
		                        &\lesssim& t^{-N} =|\xi|^{-N}
	\end{eqnarray*}
	if $\Phi$ is sufficiently smooth. On the other hand, let $\xi=te$ with $e$ a unit vector with $|e_2|/|e_1|\geq 3/2$ and $t>0$. There exists a uniform $\delta>0$ such that
	for all $\eta$ with $|\eta|<\delta t$ we have $\xi - \eta \in \C$ and hence $\chi_\C(\xi - \eta) = 1$. Now we can estimate
	\begin{eqnarray*}
		|\hat p_1(\xi)| &=& |1-\hat p_0(\xi)| =|\int_{\R^2}\hat\Phi(\eta)(1-\chi_\C(\xi - \eta))d\eta | \\
		                        &=& |\int_{|\eta|>\delta t}\hat\Phi(\eta)(1-\chi_\C(\xi - \eta))d\eta | \lesssim t^{-N}=|\xi|^{-N}
	\end{eqnarray*}
	again for $\Phi$ smooth. Note that in the first equality we have used that $\Phi(0)=1$.
	This proves the statement.
\end{proof}

\begin{lemma}\label{lem:dec2}
	We have 
	\begin{eqnarray}\label{eq:pdecay}\nonumber
		|\hat \varphi_0(\xi)|^2& \lesssim & |\xi |^{-N}\quad \mbox{for }|\xi_1|/|\xi_2| \le 3/2 \\  
		|\hat \varphi_1(\xi)|^2& \lesssim & |\xi |^{-N}\quad \mbox{for }|\xi_2|/|\xi_1| \le 3/2 
	\end{eqnarray}
\end{lemma}

\begin{proof}
	This follows from \cite[Lemma 4.7]{Grohs2010a}. For the convenience of the reader we present a proof here as well.
	We only prove the assertion for $\varphi_0$ since the proof of the corresponding statement for $\varphi_1$ is the same.
	By definition we have
	\begin{eqnarray*}
		|\hat \varphi_0(\xi)|^2 &=&\big(\int_{a\in\R,\  |s|>2} |\hat \psi \big(a\xi_1, \sqrt{a}(\xi_2 - s\xi_1)\big)|^2 a^{-3/2}da ds \\
		&&+\int_{a>1,\ |s| < 2}|\hat \psi \big(a\xi_1, \sqrt{a}(\xi_2 - s\xi_1)\big)|^2 a^{-3/2}da ds\big).
	\end{eqnarray*}
	We start by estimating the second integral using the Fourier decay in the first variable:
	\begin{eqnarray*}
		\int_{a>1,\ |s| < 2}|\hat \psi \big(a\xi_1, \sqrt{a}(\xi_2 - s\xi_1)\big)|^2 a^{-3/2}da ds
		&\lesssim&4  \int_{a>1}(a|\xi_1|)^{-2L_1}a^{-3/2}da \\
		&\lesssim & |\xi_1|^{-2L_1}\lesssim |\xi|^{-2L_1}
	\end{eqnarray*}
	for all $\xi$ in the cone with slope $3/2$.
	To estimate the other term we need the moment condition and the decay in the second variable. We write $\hat \psi (\xi)=\xi_1^M\hat\theta(\xi)$
	and $\xi=(\xi_1,r\xi_1),\ |r|\le 3/2$.
	We start with the high frequency part:	
	\begin{eqnarray*}
		&&\int_{a <1,\ |s| > 2}|\hat \psi \big(a\xi_1, \sqrt{a}(\xi_2 - s\xi_1)\big)|^2 a^{-3/2}da \\
		&=&\int_{a <1,\ |s| > 2}|a\xi_1|^{2M}|\hat \theta \big(a\xi_1, \sqrt{a}(\xi_2 - s\xi_1)\big)|^2 a^{-3/2}da ds \\
		&\lesssim &\int_{a <1,\ |s| > 2}|a\xi_1|^{2M}|\sqrt{a}(\xi_2 - s\xi_1)|^{-2L_2} a^{-3/2}da ds \\
		&=&\int_{a <1,\ |s| > 2}a^{2M-L_2-3/2}|\xi_1|^{2M-2L_2}|r-s|^{-2L_2}da ds 
	\end{eqnarray*}
	we now use that $|r-s|$ is always strictly away from zero.
	By assumption $L_2=2M-1/2 -\varepsilon$ for some $\varepsilon >0$. 
	Hence we can estimate further
	\begin{eqnarray*}
		\dots &=&|\xi_1|^{-2(L_2-M)}\int_{a <1,\ |s| > 2}a^{-1+\varepsilon}|r-s|^{-2L_2}da ds\lesssim |\xi|^{-2(L_2-M)}.
	\end{eqnarray*}
	The low-frequency part can simply be estimated as follows:
	\begin{eqnarray*}
		&&\int_{a >1,\ |s| > 2}|\hat \psi \big(a\xi_1, \sqrt{a}(\xi_2 - s\xi_1)\big)|^2 a^{-3/2}da ds\\
		&\lesssim & |\xi_1|^{-2L_2}\int_{a >1,\ |s| > 2}a^{-3/2-L_2}|r-s|^{-2L_2}dads\\
		&\lesssim &|\xi|^{-2L_2}.
	\end{eqnarray*}
	Putting these estimates together proves the statement.
\end{proof}
\begin{lemma}\label{lem:comp}
	Assume that $\psi$ is compactly supported with support in the ball $\mathcal{B}_A:=\{\xi: |\xi|\le A\}$. Then the tempered distributions
	$(|\hat\varphi_i|^2)^\lor$, $i=0,1$ are both of compact support with support in the ball $2\sqrt{3+\sqrt{5}}\mathcal{B}_A$. 
\end{lemma}
\begin{proof}
	Since the inverse Fourier transform of the constant function $C_\psi$ is a Dirac,
	this follows if we can establish that the functions $\Delta_\psi$ and $\Delta_{\psi^\nu}$ are Fourier transforms of 
	distributions of compact support. We show this only for
	$\Delta_\psi$, the other case being similar.  
	In what follows we will use the notation $f^-(x):= f(-x)$ for a function $f$. 
	Since $\psi$ has $M$ anisotropic moments we can write $\psi = \big(\frac{\partial}{\partial x_1}\big)^M\theta$ for some
	$\theta \in L^2(\R^2)$ with the same support as $\psi$.
	Consider the function $\Theta(\xi):=|\hat \theta(\xi)|^2$. It is easy to 
	see that this is the Fourier transform of the so-called Autocorrellation function $\theta\ast \theta^-$ of $\theta$ which is compactly supported
	with support in $\mathcal{B}_{2A}$. Therefore, by the theorem of Paley-Wiener, the function $\Theta$ possesses an analytic extension to $\mathbb{C}^2$ which 
	we will henceforth 
	call $F$. Furthermore, by the same theorem $F$ is of exponential type: 
	\begin{equation}
		|F(\zeta)| \lesssim (1+|\zeta|)^K \exp(2A|\Im \zeta |),
	\end{equation}
	where $\zeta \in \mathbb{C}^2$ and $K\in \N$. 
	We now consider the analytic extension $\Gamma$ of the function $\Delta_\psi$ which is given by
	$$
		\Gamma(\zeta) = \int_{-2}^{2}\int_0^1(a\zeta_1)^{2M} F(a\zeta_1,a^{1/2}(\zeta_2-s\zeta_1)) a^{-3/2}da ds ,\quad \zeta \in \mathbb{C}^2.
	$$
	Since $M\geq 1$ by (\ref{eq:momsm}) the above integral is locally integrable which implies that $\Gamma$ is actually
	an entire function. 
	It is also of exponential type: Writing $M_{as}:=\left(\begin{array}{cc}a&0\\ -s a^{1/2} & a^{1/2}\end{array}\right)$ a short computation reveals that
	\begin{equation*}
		\|M_{as}\|_{L^2(\mathbb{C}^2)\to L^2(\mathbb{C}^2)} \le a^{1/2}\big(1 + \frac{s^2}{2} + (s^2 + \frac{s^2}{4})^{1/2}\big)^{1/2}=: a^{1/2}C(s). 
	\end{equation*}
	Now we estimate
	\begin{eqnarray*}\nonumber
		|\Delta_\psi(\zeta) |&=&| \int_{-2}^{2}\int_0^1 (a\zeta_1)^{2M} F(M_{as}\zeta) a^{-3/2}da ds| \lesssim |\zeta|^{2M} (1+|M_{as}\zeta|)^{K}\exp(|M_{as}\Im\zeta|)\\
		                                 &\lesssim&\sup_{s\in [-2,2]}(1+|\zeta|)^{2M+K}\exp(2AC(s)|\Im\zeta|)\lesssim (1+|\zeta|)^{2M+K}\exp(2AC(2)|\Im\zeta|).
	\end{eqnarray*}
	By the Theorem of Paley-Wiener-Schwartz \cite{Hormander1983} it follows that $\Delta_\psi$ is of compact support with support in $\mathcal{B}_{2\sqrt{3+\sqrt{5}}A}$.
	\end{proof}

\subsection{Main Result}
We are now ready to prove our main result, the local representation formula in Theorem \ref{thm:main} which
is similar to (\ref{eq:tforig}) only with local frequency projections (given by convolution with $p_0,p_1$) and
with possibly compactly supported shearlets. In addition, in Theorem \ref{thm:nonex}
we show that  in a way this is the simplest representation that one can achieve with shearlets -- the local frequency projections
are necessary in order to wind up with useful systems.

From now on we shall assume that $0\le \hat \Phi(\xi)\le 1$ for all $\xi \in \R^2$ and define 
$$
	\hat q_i(\xi) := \hat p_i(\xi)^{1/2}.
$$
Furthermore, we define
$$
	\hat \varphi(\xi):= (\hat p_0(\xi)|\varphi_0(\xi)|^2 + \hat p_1(\xi)|\varphi_1(\xi)|^2)^{1/2}.
$$
Observe that by the positivity assumption above the radicands in the previous definitions are nonnegative and real.
\begin{theorem}\label{thm:main}
	We have the representation formulas
	\begin{eqnarray}\nonumber
		\|f\|_2^2 &=& \frac{1}{C_\psi}\big(\int_{\R^2}\int_{-2}^2\int_0^1 |\langle f , q_0\ast \psi_{ast}\rangle|^2 a^{-3}da ds dt \\\nonumber
		               &+& \int_{\R^2}\int_{-2}^2\int_0^1 |\langle f , q_1\ast \psi_{ast}^\nu\rangle|^2 a^{-3}da ds dt \\\label{eq:repr1}
			      &+& \int_{\R^2}|\langle f , T_t  \varphi\rangle|^2dt \big)
	\end{eqnarray}
	and
	\begin{eqnarray}\nonumber
	       	f &=& f^{\mbox{high} }+f^{\mbox{low}}:=\frac{1}{C_\psi}\big(\int_{\R^2}\int_{-2}^2\int_0^1 \langle f ,  \psi_{ast}\rangle p_0\ast \psi_{ast} a^{-3}da ds dt \\\nonumber
		   &+& \int_{\R^2}\int_{-2}^2\int_0^1 \langle f ,\psi_{ast}^\nu\rangle p_1\ast \psi_{ast} a^{-3}da ds dt\big) \\\label{eq:repr2}
			      &+&\frac{1}{C_\psi}\big( \int_{\R^2}\langle f , T_t \varphi_0\rangle p_0\ast T_t\varphi_0 dt +\int_{\R^2}\langle f , T_t  \varphi_1\rangle p_1\ast T_t \varphi_1dt\big).
	\end{eqnarray}
	The function $f^{\mbox{low}}$ satisfies
	\begin{equation}\label{eq:smoothwin}
		(f^{\mbox{low}})^\land(\xi) \lesssim |\xi|^{-N}
	\end{equation}
	for any $f$. We have the following \emph{locality property}: If $\psi$ is compactly supported in $\mathcal{B}_A$ and $\Phi$ has support in $\mathcal{B}_B$, then  
	$f^{\mbox{low}}(t)$ only depends on $f$ restricted to 
	$t+\mathcal{B}_{2\sqrt{3+\sqrt{5}}A+ B}$.
\end{theorem}
\begin{proof}
	We first prove (\ref{eq:repr2}). Taking the Fourier transform of the right hand side yields
	$$
		\frac{1}{C_\psi}\big(\hat p_0(\xi)(\Delta_\psi(\xi) + |\hat\varphi_0(\xi)|^2)+\hat p_1(\xi)(\Delta_{\psi^\nu}(\xi) + |\hat\varphi_1(\xi)|^2)\big) \hat f(\xi) =\hat f(\xi).
	$$
	The proof of (\ref{eq:repr1}) is similar.
	We prove the equation (\ref{eq:smoothwin}): Again taking the Fourier transform of $f^{\mbox{low}}$ gives
	$$
		\frac{1}{C_\psi}\big(\hat p_0(\xi)|\hat\varphi_0(\xi)|^2 + \hat p_1(\xi)|\hat\varphi_1(\xi)|^2\big)\hat f(\xi). 
	$$
 	Now, the desired estimate follows from Lemmas \ref{lem:dec1} and \ref{lem:dec2}.
	The last statement follows from the observation that $f^{\mbox{low}}$ can be written as
	$$
		f^{\mbox{low}} =\frac{1}{C_\psi} f \ast \big(p_0 \ast (|\hat\varphi_0|^2)^\lor + p_1 \ast (|\hat\varphi_1|^2)^\lor\big)  
	$$
	together with Lemma \ref{lem:comp}.
\end{proof}
\begin{remark}While we could show that the functions $\varphi_i^- \ast \varphi_i$ and $p_i$, $i=1,2$ are compactly supported,
	it would be desirable to show that also the functions $\varphi_i , \ q_i,\ i=1,2$ have compact support. We currently do not know
	how to do this. By a result of Boas and Kac \cite{boas1945}, in the univariate case there always exists for any compactly supported function
	$g$ with positive Fourier transform a compactly supported function
	$f$ with $f\ast f^- = g$. However, in dimensions $\geq 2$ this holds no longer true and things become
	considerably more difficult.
\end{remark}
The previous theorem for the first time gives a completely local representation for square integrable functions which also allows to handle directional phenomena
efficiently: by the results of \cite{Grohs2010a} it follows that the decay rate of the coefficients $\langle f , \psi_{ast}\rangle$ for $a\to 0$ accurately describes the 
microlocal smoothness of $f$ at $t$ in the direction with slope $s$. 

It is interesting to ask if the frequency projections given by convolution with $p_0,p_1$ are really necessary, or in other words if it is possible
to construct tight frame systems $(T_t \varphi)_{t\in \R^2}\cup (\psi_{ast})_{a\in [0,1], s\in [-2,2], t\in \R^2}
\cup (\psi^\nu_{ast})_{a\in [0,1], s\in [-2,2], t\in \R^2}$ for $L^2(\R^2)$. We show that this is 
actually impossible, meaning that in a sense Theorem \ref{thm:main} is the best we can do. 

We remark that for any shearlet $\psi$, the constant $C_\psi$ can also be computed as
$$
	C_\psi=\int_{-\infty}^{\infty}\int_0^\infty |\hat\psi(a\xi_1,a^{1/2}(\xi_2-s\xi_1))|^2 a^{-3/2}dads,
$$
as a short computation reveals. This fact is related to the inherent group structure of the shearlet transform \cite{Dahlke2008}.
\begin{theorem}\label{thm:nonex}
	Assume that $\psi=  \big(\frac{\partial}{\partial x_1}\big)^M\theta$ is a shearlet with $M\geq 1$ vanishing moments in $x_1$-direction such that either
	$M>1$ or $M=1$ and $\hat\theta\in L^\infty(\R^2)$. Furthermore we assume and $L_1>0$, $L_2>M$ with $L_1,L_2$ defined as in Lemma \ref{lem:dec2}. 
	Then there does not exist a window function $\varphi$ such that 
	$$
		\lim_{\xi \to \infty}\hat \varphi (\xi) = 0
	$$
	and such that the system 
	 $(T_t \varphi)_{t\in \R^2}\cup (\psi_{ast})_{a\in [0,1], s\in [-2,2], t\in \R^2} \cup (\psi^\nu_{ast})_{a\in [0,1], s\in [-2,2], t\in \R^2}$ constitutes 
	 a tight frame for $L^2(\R^2)$, which means
	 that a representation formula
	 \begin{eqnarray}\nonumber
		\|f\|_2^2 &=& \frac{1}{C}\big(\int_{\R^2}\int_{-2}^2\int_0^1 |\langle f ,  \psi_{ast}\rangle|^2 a^{-3}da ds dt \\
		               &+& \int_{\R^2}\int_{-2}^2\int_0^1 |\langle f , \psi_{ast}^\nu\rangle|^2 a^{-3}da ds dt + \int_{\R^2}|\langle f , T_t  \varphi\rangle|^2dt\big) 
		               		               \label{eq:reprnonex}
	\end{eqnarray}
	 holds for all $f\in L^2(\R^2)$ and some constant $C$.
\end{theorem}
\begin{proof}
	In terms of the Fourier transform, (\ref{eq:reprnonex}) translates to 
	$$
		\Delta_\psi(\xi) + \Delta_{\psi^{\nu}}(\xi) + |\hat\varphi(\xi)|^2 = C.
	$$
	If we assume that $\varphi$ has Fourier decay $\lim_{|\xi|\to \infty}\hat \varphi(\xi) = 0$
	this would imply that 
	$$
		\lim_{|\xi|\to \infty}\big(\Delta_\psi(\xi) + \Delta_{\psi^{\nu}}(\xi)\big) = C.
	$$
	Lemma \ref{lem:dec2} implies that $\lim_{|\xi|\to \infty}\Delta_\psi(\xi) = \lim_{|\xi|\to \infty}\Delta_{\psi^{\nu}}(\xi)=C_\psi$
	for $\frac23\le \frac{|\xi_1|}{|\xi_2|} \le \frac32$. It follows that $C=2C_\psi$. On the other hand, using the moment
	condition $\hat \psi = \xi_1^M\hat\theta$ for $M=1$ and $\hat\theta \in L^\infty(\R^2)$, we have the following estimate
	for $\Delta_\psi$ and $\xi$ in the strip $\mathcal{S}_\delta:=\{\xi : |\xi_1|\le \delta\}$:
	$$
		|\Delta_\psi(\xi)|\le  \delta^{2}\int_{-2}^{2}\int_0^1a^{2}\|\hat\theta\|_\infty a^{-3/2}dads\le 8\|\hat \theta\|_\infty \delta^2.
	$$
	If we assume that $M>1$, we can write $\hat \psi(\xi) = \xi_1^{M-1}\mu(\xi)$ where $\mu$ is still a shearlet. We get a similar estimate as above:
	$$
		|\Delta_\psi(\xi)|\le  \delta^{2(M-1)}\int_{-2}^{2}\int_0^1 |\hat\mu(a\xi_1,a^{1/2}(\xi_2-s\xi_1))|^2 a^{-3/2}dads\le C_\mu\delta^{2(M-1)}.
	$$
	At any rate, by choosing $\delta $ appropriately small, this implies that for $\xi\in \mathcal{S}_\delta$
	we have 
	$$
		C=\lim_{|\xi|\to \infty}\big(\Delta_\psi(\xi) + \Delta_{\psi^{\nu}}(\xi)\big) < C,
	$$
	which gives a contradiction.
\end{proof}
It is easy to extend this argument to show that there do not exist shearlet frames such that there exists a dual frame
which also has the structure of a shearlet system, see \cite{Ali1993} for more information on frames
in general and \cite{Grohs10} for more information on shearlet frames in particular.
\section{Concluding Remarks}
In future work we would like to address the problem of constructing continuous tight frames for the so-called 'Hart Smith Transform' \cite{Candes2003,Smith1998} where
the shear operation is replaced with a rotation. We think that in this case the results might become simpler. One reason for this is that \
in this case no (smoothed) projection onto a frequency cone
is needed. 
In view of constructing discrete tight frames we think that a simple discretization of a continuous tight frame will not work for non-bandlimited shearlets. 
The approach that we are currently pursuing in this direction is to construct so-called Shearlet MRA's via specific scaling functions and to try to 
generalize the 'unitary extension principle' of Ron and Shen \cite{ron1997} to the shearlet setting \cite{Grohs10}.
\section{Acknowledgments}
The research for this paper has been carried out while the author was working at the Center for Geometric Modeling and Scientific Visualization at KAUST, Saudi Arabia.


\begin{thebibliography}{10}

\bibitem{Ali1993}
S.~T. Ali, J.~P. Antoine, and J.~P. Gazeau.
\newblock Continuous frames in {H}ilbert space.
\newblock {\em Annals of Physics}, 222:1--37, 1993.

\bibitem{boas1945}
R.~Boas~Jr. and M.~Kac.
\newblock {Inequalities for Fourier transforms of positive functions}.
\newblock {\em Duke Mathematical Journal}, 12(1):189--206, 1945.

\bibitem{Candes2003}
E.~J. Candes and D.~L. Donoho.
\newblock Continuous curvelet transform: I. resolution of the wavefront set.
\newblock {\em Applied and Computational Harmonic Analysis}, 19:162--197, 2003.

\bibitem{Dahlke2008}
S.~Dahlke, G.~Kutyniok, P.~Maass, C.~Sagiv, H.-G. Stark, and G.~Teschke.
\newblock The uncertainty principle associated with the continuous shearlet
  transform.
\newblock {\em International Journal of Wavelets, Multiresolution and
  Information Processing}, 6:157--181, 2008.

\bibitem{Grohs2010a}
P.~Grohs.
\newblock Continuous shearlet frames and resolution of the wavefront set.
\newblock Technical report, KAUST, 2009.
\newblock available from http://www.dmg.tuwien.ac.at/grohs/papers/shres.pdf.

\bibitem{Grohs10}
P.~Grohs.
\newblock Refinable functions for composite dilation systems.
\newblock 2009.
\newblock manuscript in preparation.

\bibitem{Hormander1983}
L.~H{\"o}rmander.
\newblock {\em The Analysis of linear Partial Differential Operators}.
\newblock Springer, 1983.

\bibitem{Labate2009}
G.~Kutyniok and D.~Labate.
\newblock Resolution of the wavefront set using continuous shearlets.
\newblock {\em Transactions of the American Mathematical Society},
  361:2719--2754, 2009.

\bibitem{Labate2005}
D.~Labate, G.~Kutyniok, W.-Q. Lim, and G.~Weiss.
\newblock Sparse multidimensional representation using shearlets.
\newblock In {\em Wavelets XI (San Diego, CA, 2005), 254-262, SPIE Proc. 5914,
  SPIE, Bellingham, WA}, 2005.

\bibitem{ron1997}
A.~Ron and Z.~Shen.
\newblock Affine systems in {$L^2(\mathbb{R}^d)$}: The analysis of the analysis
  operator.
\newblock {\em Journal of Functional Analysis}, 148(2):408 -- 447, 1997.

\bibitem{Smith1998}
H.~F. Smith.
\newblock A {H}ardy space for {F}ourier integral operators.
\newblock {\em Journal of Geometic Analysis}, 8:629--653, 1998.

\end{thebibliography}
\end{document}